\documentclass[a4paper,12pt]{article}
\usepackage{amsthm}
\usepackage{extpfeil}
\usepackage{mathrsfs}
\usepackage{graphicx}

\newlength{\abstractwidth}
\numberwithin{equation}{section}
\renewcommand{\thanks}[1]{\footnote{#1}} % Use this for footnotes

\newcommand{\be}{\begin{equation}}
\newcommand{\bea}{\begin{eqnarray}}
\newcommand{\eea}{\end{eqnarray}} \newcommand{\ee}{\end{equation}}
 
 \def\ba{\begin{eqnarray}}
\def\ea{\end{eqnarray}}

\def\B{{\cal B}}
\def\D{{\cal D}}

\def\K{{\cal K}}

\def\cX{{\cal X}}
\def\cL{{\cal L}}

\def\r{\rho}

\def\cS{{\cal S}}

\def\fg{\mathfrak{g}}
\def\ft{\mathfrak{t}}

\def\o{\omega}

\def\det{{\rm det}}

\def\log{\,{\rm log}\,}
\def\exp{\,{\rm exp}\,}

\def\o{\omega}

\def\al{\alpha}
\def\b{\beta}
\def\g{\gamma}
\def\d{\delta}

\def\l{\lambda}

\def\o{\omega}

\def\r{\rho}
\def\si{\sigma}
\def\t{\theta}

\def\D{\Delta}
\def\O{\Omega}
\def\ph{\varphi}
\def\L{\Lambda}

\def\k{\kappa}

\def\ov{\overline}
\def\ti{\tilde}

\def\i{\infty}

\def\B{{\cal B}}

\def\D{\Delta}

\def\cI{{\cal I}}

\def\cF{{\cal F}}

\def\cW{{\cal W}}

\def\K{{K\"ahler\ }}

\def\[{{\bf [}}
\def\]{{\bf ]}}

\def\cpx{\mathbb{C}}
\def\proj{\mathbb{P}}
\def\real{\mathbb{R}}

\def\pP{{\partial P}}
\def\pa{\partial}
\def\opa{\overline{\partial}}

\def\cV{\mathcal{V}}

\def\cP{\mathcal{P}}
\def\cE{\mathcal{E}}

\newtheorem{theorem}{Theorem}
\newtheorem{proposition}{Proposition}[section]
\newtheorem{lemma}[proposition]{Lemma}

\newtheorem{definition}[proposition]{Definition}

\begin{document}

\begin{normalsize}

\begin{center}
    {\large \bf Scalar Curvature and Stability of Toric Fibrations
        \footnote{Work supported in part by National Science Foundation grant
                    DMS-07-57372.}}

    \bigskip

    {Thomas Nyberg} \\

    \medskip

    \today

    \smallskip

\end{center}

\medskip

\begin{abstract}
    We study fibrations $\cV$ of toric varieties over the flag variety $G/T$,
    where $G$ is a compact semisimple Lie group and $T$ is a maximal torus. From
    symplectic data, we construct test configurations of $\cV$ and compute their
    Futaki invariants by employing a generalization of Pick's Theorem. We also
    give a simple form of the Mabuchi Functional.
\end{abstract}

\section{Introduction}

A major open problem in complex geometry is to find algebro-geometric
``stability" conditions on a complex manifold which imply the existence of
constant scalar curvature \K metrics. For an overview of the status of the
general problem, the reader is refered to Phong and Sturm's work
\cite{phong-sturm}. There have been many advances in the general problem, but an
especially fruitful area of research has been the analysis of this question
within the framework of toric varieties.
%
%----------------------------NEW LINE---------------------------%

%----------------------------NEW LINE---------------------------%
%
Guillemin \cite{guil} and Abreu \cite{abreu} were able to transform the
differential geometry of an $n$-dimensional toric variety $V$ to that of its
moment polytope $\ov{P}$ (following Donaldson's conventions, we denote by $P$
the interior of the moment polytope).  Instead of studying toric \K metrics
directly, one studies their corresponding \textit{symplectic potentials}
$u$---continuous convex functions defined on $\ov{P}$ which are smoooth on $P$.
Abreu found the equation for the scalar curvature $S$ in terms of $u$ to be
\begin{equation*}
    \label{}
    S(u) = -(u^{jk})_{jk},
\end{equation*}
where $u^{jk}$ is the inverse of the Hessian $u_{jk}$ of $u$. In the series of
papers \cite{don-02, don-int, don-cont1, don-cont2}, Donaldson uses the
Abreu equation to develop the theory of K-stabity of toric varities, and shows
that all K-stable toric surfaces admit toric metrics of constant scalar
curvature.
%
%----------------------------NEW LINE---------------------------%

%----------------------------NEW LINE---------------------------%
%
In \cite{don-02}, Donaldson defines an algebraic version of the Futaki invariant
for test configurations of complex manifolds which he uses in his definition of
K-stability. In the toric setting, he shows that test configurations can be
constructed from piecewise-linear rational functions on $P$ and that the
Donaldson-Futaki invariant is given as a linear functional on these functions.
%
%----------------------------NEW LINE---------------------------%

%----------------------------NEW LINE---------------------------%
%
In \cite{podesta-spiro} Podest\`{a} and Spiro studied fibrations of toric
varieties over a flag variety base by studying the fiber product $\cV := G
\times_T V,$ where $G$ is a compact semisimple Lie group and $T \subset G$ is a
maximal torus. In \cite{don-sym} Donaldson suggests studying the constant scalar
curvature problem on these spaces. Building off the work of Raza \cite{raza},
Donaldson gives the scalar curvature $S$ of a toric metric on such a fibration
as
\begin{equation}
    \label{abreu-raza}
    S(u) = -W^{-1} (W u^{jk})_{jk} + f_G,
\end{equation}
where $W$ is the Duistermaat-Heckman polynomial and $f_G$ is a smooth
function---both functions actually only depend on $G$.

In this paper, we extend the theory of K-stability to this setting. We
generalize both the construction of the test configurations and the formula for
the Donaldson-Futaki invariant to this setting. We fix a positive line bundle
$L$ over $V$ and an action of $T$ on $L$ such that the line bundle $\mathscr{L}
:= G \times_T L$ over $\cV$ is positive. This means that any \K metric in
$c_1(\mathscr{L})$ has the same average scalar curvature $a$. The first result
is the following:
%
%----------------------------NEW LINE---------------------------%

%----------------------------NEW LINE---------------------------%
%
\begin{theorem}
    \label{main_theorem}
    Given any rational piecewise linear function $f$ on $\ov{P}$, there exists a
    test-configuration $\cX$ for $\cV$ with Futaki-Invariant $F_1$ given by
    \begin{equation*}
        \label{}
        F_1 = -\frac{1}{2\mathrm{Vol}_W(P)} \left( \int_P ff_GWd\mu +
            \int_{\partial P} fWd\si - a\int_P fWd\mu \right),
    \end{equation*}
    where $d\mu$ is the Lebesgue measure, $d\si$ is a measure on $\pP$ defined
    in Definition \ref{sigma_def}, and $\mathrm{Vol}_W(P) = \int_P Wd\mu.$
\end{theorem}
%
%----------------------------NEW LINE---------------------------%

%----------------------------NEW LINE---------------------------%
%
A natural starting point when trying to solve the scalar curature equation is
to consider the Mabuchi Functional. By analyzing \eqref{abreu-raza}, we derive
the following:
\begin{theorem}
    \label{mabuchi_theorem}
    The Mabuchi Functional $\cF$ defined on symplectic potentials $u$, is given
    by the mapping
    \begin{equation*}
        \label{}
        u \mapsto -\int_P \log \det(u_{jk}) Wd\mu + 2\int_\pP uWd\si - \int_P
            uAWd\mu,
    \end{equation*}
    where $A = \frac{a-f_G}2$.
\end{theorem}
%
%----------------------------NEW LINE---------------------------%

%----------------------------NEW LINE---------------------------%
%
The proof of Theorem \ref{mabuchi_theorem} is a straightforward generalization
of the methods in \cite{don-02} once one undestands the geometry of the spaces
involved. Theorem \ref{main_theorem}, however, requires comparing the
asymptotics of certain sums over lattice points of the scaled polytope
$k\ov{P}$. A key technical step in the proof requires the following
generalization of Pick's theorem which may be interesting in its own right:
\begin{lemma}
    \label{sigma_lemma}
    Let $P \subset \real^n$ be an integer polytope and let $h$ be a convex
    function in $C^2(\ov{P})$.  Then we have
    \begin{equation}
        \label{lemma_sum}
        \sum_{p \in \ov{P} \cap \frac{1}{k}\mathbb{Z}^n}h(p) = \left( \int_P
            hd\mu \right)k^n + \left( \frac{1}{2}\int_{\partial P}hd\si \right)
            k^{n-1} + o(k^{n-2}).
    \end{equation}
\end{lemma}
\noindent The proof of this lemma is left until the end in Section
\ref{section_lemma_proof}. Once equipped with this Lemma, careful computations
leads one to conclude that the Weyl Dimension Formula from classical Lie theory
essentially agrees with the function $W$ to highest order and with $Wf_G$ to
second-heighst order, which allows us to relate these asymptotic sums to the
scalar curvature equation.
%
%----------------------------NEW LINE---------------------------%

%----------------------------NEW LINE---------------------------%
%
The outline of this paper is as follows. In Section 2 we give some background
about the K-stability of complex manifolds. In Section 3 we describe the spaces
that will be studied and review the Lie algebra theory we will need. In Section
4 we give a derivation of the scalar curvature equation \eqref{abreu-raza}. In
Section 5 we explain how to construct test configurations from piecewise linear
functions and then in Section 6 we compute the Futaki invariant of these test
configurations. In Section 7 we give the formula for the Mabuchi functional on
the polytope $P$. Finally in Section 8 we give the proof of a generalization of
Pick's theorem which is used in the computation of the Futaki invariant.
%
%----------------------------NEW LINE---------------------------%

%----------------------------NEW LINE---------------------------%
%
\textbf{Acknowledgements:} Many of my fellow graduate students and members of
the faculty have patiently helped while I prepared this paper, but I would
especially like to thank Anna Pusk\'{a}s and Tristan Collins for their help, and
Professor Chiu-Chu Liu for always making herself available for questions. I
would also like to thank my advisor Professor Phong whose support and advice has
been indispensible during my studies.

\section{Background}

The Donaldson-Futaki invariant is an invariant assigned to any ample line bundle
$\L$ over a projective scheme $X$, such that there is a $\cpx^*$-action on the
pair $(X, \L)$. For each positive integer $k$, let $H_k = H^0(X, \L^k)$, let
$d_k = \mathrm{dim}(H_k)$, and let $w_k$ be the weight of the induced
$\cpx^*$-action on $\L^{d_k} H_k$. Write $F(k) = \frac{w_k}{kd_k}$ and note that
by general theory, $F(k)$ is a rational function for large $k$. We have the
expansion
\begin{equation}
    \label{def_futaki_invariant}
    F(k) = F_0 + F_1k^{-1} + \cdots,
\end{equation}
for large enough $k$. The Donaldson-Futaki invariant of $(X, \L)$ is the
rational number $F_1$ in this expansion. See \cite{don-02} for more details.
%
%----------------------------NEW LINE---------------------------%

%----------------------------NEW LINE---------------------------%
%
In order to associate a Donaldson-Futaki invariant to a compact complex manifold
$M$ with ample line bundle $L$, one needs to associate a pair $(X, \L)$ to $(M,
L)$. To this end, Donaldson defines a \textit{test configuration}---a kind of
algebraic degeneration of $(M, L)$. A test configuration is a scheme $\cX$ with
a $\cpx^*$-action, a $\cpx^*$-equivariant line bundle $\cL \to \cX$, and a flat
$\cpx^*$-equivariant map $\pi : \cX \to \cpx$, where $\cpx^*$ acts on $\cpx$ by
standard multiplication. Furthermore, for any fiber $\cX_p = \pi^{-1}(p),$ where
$p \neq 0$, we require that $(\cX_p, \cL|_{\cX_p})$ be ismorphic to $(M, L)$.
%
%----------------------------NEW LINE---------------------------%

%----------------------------NEW LINE---------------------------%
%
Now let $(\cX, \cL)$ be a test configuration for $(M, L)$. If we let $\cX_0$ be
the restriction of $\cX$ to the fiber over 0 and let $\cL_0$ be the restriction
of $\cL$ to $\cX_0$, then $(\cX_0, \cL_0)$ has a well-defined Donaldson-Futaki
invariant. We define the Donaldson-Futaki invariant of the test configuration
$(\cX, \cL)$ to be the Donaldson-Futaki invariant of $(\cX_0, \cL_0)$. The pair
$(M,L)$ is defined to be \textit{K-stable} if the Donaldson-Futaki invariant of
\textit{any} test-configuration of $(M,L)$ is less than or equal to 0, with
equality if and only if the test configuration is the trivial product
configuration.
%
%----------------------------NEW LINE---------------------------%

%----------------------------NEW LINE---------------------------%
%
Thus far, these concepts are defined in general for any pair $(M, L)$. In
\cite{don-02}, Donaldson specializes these concepts to the case of
toric varieties. In this paper we extend his description to encompass the
aforedescribed toric fibrations.

\section{Fibrations of toric varieties}

\subsection{Lie Theory and Construction of $G \times_T V$}
\label{lie_theory_sec}

In this section, we describe the construction of the toric fibrations in
question and review the Lie algebra theory we need. See \cite{humphreys,
sepanski} for a more detailed treatment of the theory. Let $G$ be a semisimple
Lie group and let $T \subset G$ be a maximal torus whose dimension is $n$. Let
$\ft \subset \fg$ denote the corresponding Lie algebras. Let $\fg_\cpx = \fg
\otimes_\real \cpx$ and $\ft_\cpx = \ft \otimes_\real \cpx$ be the
complexifications of $\fg$ and $\ft$. Let $\k$ be the Killing form of
$\fg_\cpx$. Since $G$ is semisimple, $\k$ is a non-degenerate bilinear form,
which means that
\begin{equation*}
    \label{}
    \fg = \ft \oplus \ft^\perp,
\end{equation*}
where $\ft^\perp$ is the perpendicular space to $\ft$ with respect to $\k$. By
$\cpx$-linearity, we also have
\begin{equation*}
    \label{}
    \fg_\cpx = \ft_\cpx \oplus \ft_\cpx^\perp.
\end{equation*}
Let $\D \subset \ft_\cpx^*$ be the finite set of \textit{weights} of $\fg_\cpx$
and let
\begin{equation*}
    \label{}
    \fg_\cpx = \ft_\cpx \oplus \left( \bigoplus_{\al \in \D} \fg_\al \right),
\end{equation*}
be the weight space decomposition of $\fg_\cpx$. By choosing a system of positive
weights $\D^+$ and negative weights $\D^-$, we have
\begin{equation*}
    \label{}
    \fg_\cpx = \ft_\cpx \oplus \left( \bigoplus_{\al \in \D^+} \fg_\al \right)
            \oplus \left( \bigoplus_{\al \in \D^-} \fg_\al \right).
\end{equation*}
For each $\al \in \D$ there are \textit{real} elements $V_\al, W_\al, H_\al \in
(\fg_\al \oplus \fg_{-\al} \oplus [\fg_\al,\fg_{-\al}] ) \cap \fg$, such that
\begin{equation}
    \label{su(2)-triple}
    [W_\al, V_\al] = 2H_\al, \ \ \ \ \ [V_\al, H_\al] = 2W_\al, \ \ \ \ \
        [H_\al, W_\al] = 2V_\al.
\end{equation}
Furthermore, $\al(-iH_\al) = 2$, for each $\al \in \D^+$. This is the standard
$SU(2)$-triple. For example, in the case where $G = SU(2)$, we have
\begin{equation*}
    \label{}
    V =
    \begin{bmatrix}
        0 & -1 \\
        1 & 0 \\
    \end{bmatrix},
    \ \ \ \ \ 
    W =
    \begin{bmatrix}
        0 & i \\
        i & 0 \\
    \end{bmatrix},
    \ \ \ \ \ 
    H =
    \begin{bmatrix}
        i & 0 \\
        0 & -i \\
    \end{bmatrix}.
\end{equation*}
%
%----------------------------NEW LINE---------------------------%

%----------------------------NEW LINE---------------------------%
%
Let $\{ \al_1, \ldots, \al_n \} \subset \D^+$ be the set of \textit{simple
roots} and let $H_j = H_{\al_j}$ be the corresponding elements in $\ft.$
Since $H_1, \ldots, H_n$ provides a basis for $\ft$, we can define $\nu^1,
\ldots, \nu^n$---the \textit{fundamental weights}---to be the corresponding dual
basis of $\ft^*$.
%
%----------------------------NEW LINE---------------------------%

%----------------------------NEW LINE---------------------------%
%
Identify $\ft$ with $\real^n$ using the basis $H_1, \ldots, H_n$. Let
$(V,L)$ be a pair consisting of a toric variety $V$ and a line bundle $L \to V$.
This is equivalent to choosing a closed Delzant polytope $P \subset \real^n \cong
\ft^*$---which is determined up to a constant vector in $\mathbb{Z}^n$. By
fixing that constant vector, one fixes the ``linearized" action of $(\cpx^*)^n$
on $L$ which is compatible with the action on $V$. A basis for the sections of
$L$ is in one-to-one correspondance with the lattice points of $\ov{P}$. Let $\l
= (k_1, \ldots, k_n) \in \ov{P} \cap \mathbb{Z}^n$ be any such point and let
$s_\l$ be the corresponding section of $L$.  Then the action of $\b = (\b_1,
\ldots, \b_n) \in (\cpx^*)^n$ on $s_\l$ is given by
\begin{equation*}
    \label{}
    \b \cdot s_\l = \prod_{j=1}^n \b_j^{-k_j}s_\l.
\end{equation*}
Returning to our Lie group $G$, we can consider $T = (S^1)^n \subset
(\cpx^*)^n$ by way of our basis. Hence we have an action of $T$ on $V$ and a
compatible action of $T$ on $L$. This means that we can form the spaces
\begin{equation}
    \label{def_cV}
    \mathscr{L} = G \times_T L, \ \ \ \ \ \cV = G \times_T V, \ \ \ \ \ \B :=
        G/T.
\end{equation}
We have that $\mathscr{L} \to \cV$ is a line bundle with a compatible left
$G$-action and fiberwise right $T$-action. Furthermore, the projection map $\cV
\to \B$ respects these actions.
%
%----------------------------NEW LINE---------------------------%

%----------------------------NEW LINE---------------------------%
%
An important fact is that $\mathscr{L}, \cV,$ and $\B$ have holomorphic
structures. To see this, take a complexification $G_\cpx$ of $G$ and let $B$,
with $T_\cpx \subset B \subset G_\cpx$, be the Borel subgroup corresponding to
the positive roots. One has then that $G_\cpx \times_B L \cong G \times_T L$,
$G_\cpx \times_B V \cong G \times_T V$, and $G_\cpx / B \cong G / T$ as smooth
differential manifolds. The left $G$-action is a subset of the left
$G_\cpx$-action, which acts by biholomorphisms. See \cite{don-sym} for more
details.
%
%----------------------------NEW LINE---------------------------%

%----------------------------NEW LINE---------------------------%
%
The space of holomorphic sections of $\mathscr{L}$ decomposes as
$G_\cpx$-representations by
\begin{equation}
    \label{G_sections}
    H^0(G_\cpx \times_B L) = G_\cpx \times_B H^0 \left(\bigoplus_{\l \in \ov{P}}
        (L_\l) \right) = \bigoplus_{\l \in \ov{P}} \left( H^0(G_\cpx \times_B
        L_\l) \right),
\end{equation}
where the $L_\l \subset L$ is the line bundle spanned by the section $s_\l$. We
are mainly concerned with the dimension of \eqref{G_sections}.  The Borel-Weil
Theorem states that $\dim H^0(G_\cpx \times_B L_\l)$ is a polynomial in $\l$
given by Weyl dimension formula as
\begin{equation}
    \label{dimension_formula}
    \dim H^0(G_\cpx \times_B L_\l) = \prod_{\al \in \D^+} \frac{\k(\r + \l,
        \al)}{\k(\r, \al)},
\end{equation}
where $\r = \sum_i \nu^i$ is the sum of the fundamental weights.

\subsection{Metric Geometry of $G \times_T V$}

In \cite{don-sym} Donaldson explains how to extend the metric and symplectic
geometries of $V$ to $\cV$. For completeness, we will review those descriptions
here.

\textit{Complex Viewpoint}: Let $g$ be a $T$-invariant \K metric on $V$ and let
$\o$ be its $T$-invariant \K form. Assume that $\o$ lies in the class $c_1(L)$.
Let $h$ be a $T$-invariant metric on $L$ such that $\o = -i\pa \opa \log h$.
Embed $(V,L)$ as the identity fiber of $(\cV,\mathscr{L})$ over $\B$. Extend $h$
to a metric $H$ on $\mathscr{L}$ by requiring it to be invariant under the left
$G$-action. Define $\O = -i\pa \opa \log H$. $\O$ is a (1,1)-form---extending
$\o$---which is invariant under the left $G$-action and the right $T$-action.
The condition for $\O$ to be positive (i.e. a \K form) is that $\ov{P}$ must be
contained in the positive Weyl chamber---which in our basis means that $\ov{P}$
is contained in the open positive quadrant of $\real^n$. In that case, denote by
$\ti{g}$ the corresponding Hermitian metric.

\textit{Symplectic Viewpoint Without Fibrations}: First let us describe the
symplectic viewpoint of $V$ without any fibration. Define the $T$-invariant
function $\ph$ by requiring that $h = e^{-2\ph}$.  Locally on the open torus
$(\cpx^*)^n \subset V$ we have
\begin{equation}
    \label{omega_metric}
    \o = \o_{\ov k j} idz^j \wedge d\ov{z}^k = 2\frac{\pa^2 \ph}{\pa z^j \pa
        \ov{z}^k} idz^j \wedge d\ov{z}^k.
\end{equation}
Define log-coordinates $(w^1, \ldots, w^n, \t^1, \ldots, \t^n)$ on $(\cpx^*)^n$ by
the mapping $z^j = \exp(w^j+i\t^j) = e^{w^j + i\t^j}$. Let $\mu: V \to \real^n$
be the corresponding moment map which in $(w,\t)$-coordinates satisfies
\begin{equation*}
    \label{}
    d(\mu_k d \t^k) = \frac{\pa \mu_k}{\pa w^j} dw^j \wedge d\t^l.
\end{equation*}
Define
\begin{equation}
    \label{phi_def}
    \phi(w^1, \cdots, w^n, \t^1, \ldots, \t^n) = \ph(e^{w^1+ i\t^1}, \ldots,
        e^{w^n+i\t^n}).
\end{equation}
We have
\begin{equation*}
    \label{}
    \exp^*(\o) = \frac{\pa^2 \phi}{\pa w^j \pa w^k} dw^j \wedge d\t^k =
        \frac{\pa \mu_k}{\pa w^j}dw^j \wedge d\t^k.
\end{equation*}
Hence $\frac{\pa \phi}{\pa w^k} = \mu_k$ up to a constant. By adjusting $h$, we
can assume that this constant is 0. Hence in the $w$-coordinates, $\mu$ is just
the gradient map of $\phi$. Define $x^j = \mu_j = \frac{\pa \phi}{\pa w^j}$ to
be the momentum coordinates on $P$. Let $u$ be the Legendre transorm of $\phi$
on $P$, then the push-forward (as a function) of $\frac{\pa^2 \phi}{\pa w^j \pa
w^k}$ equals $u^{jk}$. But the push forward of $dw^j$ is $u_{jl}dx^l$. Hence the
symplectic form in $(x,\t)$-coordinates on $T \times P$ is given by
\begin{equation}
    \label{symplectic_omega}
    \o = \sum_j dx^j \wedge d\t^j.
\end{equation}
\textbf{Remark:} In the preceeding equation, a summation symbol was used for
clarity. Einstein notation is used as much as possible, but due to the many
places where both vector fields and their dual forms are used, it is difficult
to stick to the convention of summing paired lower and upper indices. In these
circumstnces, summation symbols are used which hopefully minimizes confusion.
%
%----------------------------NEW LINE---------------------------%

%----------------------------NEW LINE---------------------------%
%
Next we would like to understand the complex structure on this space. The
complex structure $J$ sends $\frac{\pa}{\pa w^j}$ to $\frac{\pa}{\pa \t^j}$.
Under $\mu$, this vector field gets sent to
\begin{equation*}
    \label{}
    \mu_* \left( \frac{\pa}{\pa w^j} \right) = \frac{\pa \mu_k}{\pa w^j}
        \frac{\pa}{\pa x^k} = u^{jk}\frac{\pa}{\pa x^k}.
\end{equation*}
This means that the complex structure sends $\frac{\pa}{\pa x^j}$ to
$u_{jk}\frac{\pa}{\pa \t^k}.$ The Riemannian metric $g$ satisfies $g(\cdot,
\cdot) = \o( \cdot, J(\cdot))$. Hence $g(\frac{\pa}{\pa x^j}, \frac{\pa}{\pa
x^k}) = \o(\frac{\pa}{\pa x^j}, u_{kl}\frac{\pa}{\pa \t^l}) = u_{jk}$.
Similarly, $g(\frac{\pa}{\pa \t^j}, \frac{\pa}{\pa \t^k}) = u^{jk}$. This means
as a Riemannian manifold, the metric on $T \times P$ is given by
\begin{equation}
    \label{g_T_x_P}
    g = u_{jk} dx^j \otimes dx^k + u^{jk} d\t^j \otimes d\t^k.
\end{equation}
%
%----------------------------NEW LINE---------------------------%

%----------------------------NEW LINE---------------------------%
%
\textit{Symplectic Viewpoint With Fibrations}: Extend the moment map $\mu$ to
$\ti{\mu}: G \times_T V \to \real^n$ by left $G$-invariance. I.e.
$\ti{\mu}([g : z]) = \mu(z)$ for any $g \in G$ and $z \in V$. Note that the
fundamental weights $\nu^j$ can be extended to left $G$-invariant 1-forms on $G$
and that $\nu^j|_T = d\t^j$. This means that $\O$ is given on $G \times P$ by
the form
\begin{equation}
    \label{symplectic_Omega}
    \O = d(x^j\nu^j) = dx^j \wedge \nu^j + x^jd\nu^j.
\end{equation}
As before, we need to understand the complex structure $J$ on $G \times P$.
For each $\al \in \D^+$, extend the vectors $V_\al, W_\al, H_\al$ given in
\eqref{su(2)-triple} to vector fields on $G \times P$ which are invariant under
the left $G$-action. We have that the different vector fields $V_\al$ and
$W_\al$ are linearly independent, however, each $H_\al$ can be written as a sum
\begin{equation}
    \label{M_alpha-def}
    H_\al = \sum_j M^\al_j H_j,
\end{equation}
for some non-zero vector of non-negative integers $M^\al$. Furthermore, $H_j$
extends the vector field $\frac{\pa}{\pa \t^j}$ to $G \times P$. For
notational simplicity, denote by $X_j$ the vector field $u^{jk} \frac{\pa}{\pa
x^k}$.  By explicitly computing the exponential mapping on $\fg$ one sees that
the complex structure on $G \times P$ sends $V_\al$ to $W_\al$ and that
$J(X_j) = H_j$. As before, the Riemannian metric $\ti{g}$ on $G \times P$ is
given by $\ti{g}(\cdot, \cdot) = \O(\cdot, J(\cdot)).$ By inspecting
\eqref{symplectic_Omega}, we see that $\ti{g}(\frac{\pa}{\pa x^j},\frac{\pa}{\pa
x^k}) = u_{jk}$ and $\ti{g}(H_j,H_k) = u^{jk}$. Furthermore, using the
fact that $V_\al, W_\b,$ and $\nu^j$ are all left $G$-invariant, we have 
\begin{equation*}
    \label{}
    \ti{g}(V_\al,V_\b) = x^jd\nu^j(V_\al,W_\b) = -x^j \nu^j([V_\al,
        W_\b]) = \d_{\al \b} 2x^j \nu^j(H_\al) = \d_{\al \b} 2x^j M^\al_j,
\end{equation*}
and $\ti{g}(W_\al, W_\b) = \ti{g}(V_\al,V_\b),$ by $J$-invariance. Denote by
$dH_j, dV_\al,$ and $dW_\al,$ the $G$-invariant 1-forms dual to the vector
fields $H_j, V_\al,$ and $W_\al$. Hence the Riemannian metric $\ti{g}$ on $G
\times P$ is given by
\begin{equation}
    \label{g_G_x_P}
    \ti{g} = u_{jk} dx^j \otimes dx^k + u^{jk} dH_j \otimes dH_k + 2x^j M_j^\al
        \left( dV_\al \otimes dV_\al + dW_\al \otimes dW_\al \right).
\end{equation}
If we write this in terms of $X_j$ instead of $\frac{\pa}{\pa x^j}$ we get
\begin{equation}
    \label{g_G_x_Pv2}
    \ti{g} = u^{jk} dX_j \otimes dX_k + u^{jk} dH_j \otimes dH_k + 2x^j M_j^\al
        \left( dV_\al \otimes dV_\al + dW_\al \otimes dW_\al \right).
\end{equation}

\section{Scalar Curvature Equation}

Donaldson gives equation \eqref{abreu-raza} in \cite{don-sym}, but does not
provide a proof. For our purposes, the exact form of the equation is quite
important and hence we work it out explicitly. Furthermore, the proof
illucidates the relations between the real and complex geometry and is worth
providing.
%
%----------------------------NEW LINE---------------------------%

%----------------------------NEW LINE---------------------------%
%
In local holomorphic coordinates, the scalar curvature $S$ of
$\ti{g}$ is
\begin{equation*}
    \label{}
    S = -\ti{g}^{j \ov{k}} \frac{\pa^2}{\pa z^j \pa \ov{z}^k} (\log
        \det(\ti{g}_{\ov{b}a})),\end{equation*}
but it is difficult to give explicit holomorphic coordinates in terms of the
real geometry on $G \times P$. The operator $-\ti{g}^{j\ov{k}} \frac{\pa^2}{\pa
z^j \ov{z}^k} = \frac 12 \D_{\ti{g}}$---the Riemannian Laplacian. On $G \times P$,
we have the vector fields $\frac{\pa}{\pa x^j}, H_j, V_\al, W_\al$ from the last
section. We can write $\D_{\ti{g}}$ in the frame given by these fields. However,
we still need to write the function $\log \det(\ti{g}_{\ov{b}a})$ in terms more
compatible with these fields. If we let $\chi$ be a local, non-vanishing,
holomorphic $(N+n,0)$-form on $G \times P$, and let $\eta = \chi \wedge
\ov{\chi}$, then $\frac{\O^{N+n}}{\eta}$ is a smooth function and
\begin{equation*}
    \label{}
    S = \frac 12 \D_{\ti{g}} (\log \det (\ti{g}_{\ov{b}a})) = \frac 12
        \D_{\ti{g}} \left( \log \left| \frac{\O^{N+n}}{\eta} \right| \right).
\end{equation*}
As long as we can express $\O, \eta,$ and $\D_{\ti{g}}$ in terms compatible with
these fields, this will be in a form that can be readibly understood on $G
\times P$.

\subsection{Finding $\eta$}

In order to find a candidate for $\chi$, we will first find a holomorphic
$(N+n,0)$ vector field. Note that on $G \times_T (\cpx^*)^n$, the holomorphic
vector fields have nothing to do with any specific metric $\ti{g}$. However, we
can use the fact that the metric $\ti{g}$ is \K to find $\chi$. To simplify the
computations we may assume that our original \K form equals $\o_E = \sum_{j=1}^n
idz^j \wedge d\ov{z}^j$ is the standard Euclidian metric. In that case,
$\phi_E(w_1, \dots, w_n) = \frac12 (e^{2w_1}, \ldots, e^{2w_n})$ and the moment
map $D\phi_E : G \times (\real^+)^n \to (\real^+)^n$ is given by $D\phi_E =
(e^{2w_1}, \ldots, e^{2w_n}) = (y_1, \dots, y_n).$ (The coordinates are chosen
as $(y_1, \ldots, y_n)$ to stress that we are no longer working on the original
polytope $P$.) The Legendre transform of $\phi_E$ is given by $u_E(y_1, \ldots,
y_n) = \frac12 \sum_{j=1}^n (y^j \log(y^j) - y^j).$ The Hessian of $u_E$ is
given by the diagonal matrix $H(u_E) = \mathrm{Diag}(\frac{1}{2y^1}, \ldots,
\frac{1}{2y^n}).$ The moment map sends the vector field $\frac{\pa}{\pa w^j}$ to
$2y^j \frac{\pa}{\pa y^j} =: Y_j.$ The vector fields $V_\al, W_\al, H_j$ all get
sent to themselves. We have then that $J(Y_j) = H_j$ and $J(V_\al) = W_\al$. In
the frame given by $Y_j, H_j, V_\al, W_\al$, we have
\begin{equation*}
    \label{}
    \ti{g}_E = 2y^j (dY_j \otimes dY_j + dH_j \otimes dH_j) + 2y^j M^\al_j (dV_\al
        \otimes dV_\al + dW_\al \otimes dW_\al).
\end{equation*}
%
%----------------------------NEW LINE---------------------------%

%----------------------------NEW LINE---------------------------%
%
By computing the Christoffel symbols of the Levi-Civita connection $D$, one sees
that
\begin{equation}
    \label{chris1}
    D_{Y_j} Y_k = \d_{jk} Y_j, \ \ \ \ \ D_{Y_j} H_k = D_{H_k} Y_j = \d_{jk}
        H_j, \ \ \ \ \ D_{H_j} H_k = -\d_{jk}Y_j.
\end{equation}
Further computations show
\begin{equation}
    \label{chris2}
    D_{Y_k}(V_\al) = \frac{M_k^\al y^k}{\sum_j M_j^\al y^j} V_\al, \ \ \ \ \
        D_{Y_k}(W_\al) = \frac{M_k^\al y^k}{\sum_j M_j^\al y^j} W_\al.
\end{equation}
This shows that 
\begin{equation}
    \label{chris3}
    D_{(\sum_j Y_j)} (V_\al) = V_\al, \ \ \ \ \  D_{(\sum_j Y_j)}(W_\al) =
        W_\al, \ \ \ \ \ D_{(\sum_j Y_j)} H_\al = H_\al.
\end{equation}
The vector fields $H_j, V_\al, W_\b$ do not commute, and hence the computations
of the Christoffel symbols for them depend on the Lie algebra structure:
\begin{equation}
    \label{chris4}
    D_{V_\al} H_j = \frac{M_j^\al y^j}{\sum_k M_k^\al y^k} W_\al, \ \ \ \ \
        D_{W_\al} H_j = -\frac{M_j^\al y^j}{\sum_k M_k^\al y^k} V_\al.
\end{equation}
%
%
%----------------------------NEW LINE---------------------------%

%----------------------------NEW LINE---------------------------%
%
Next define smooth sections $s_j$ and $t_\al$ of the holomorphic tangent
bundle of $G \times (\real^+)^n$ by $s_j = Y_j - iH_j$ and $t_\al = V_\al -
iW_\al$. A smooth $(N+n,0)$-vector field is given by $\rho = (\bigwedge_j s_j)
\wedge (\bigwedge_\al t_\al)$. We would like to find a smooth function $f$ on $G
\times (\real^+)^n$ such that $f\rho$ is holomorphic. Since $\ti{g}_E$ is \K we
have that the Chern and Levi-Civita connections coincide. Hence we need to find
a function $f$ such that $D_{\ov{s}_j} (f \rho) = 0$ for all $j$ and
$D_{\ov{t}_\al} (f \rho) = 0$ for all $\al$.
%
%----------------------------NEW LINE---------------------------%

%----------------------------NEW LINE---------------------------%
%
Equations \eqref{chris1}-\eqref{chris4} show that $D_{\ov{s}_j} s_k =
D_{\ov{t}_\al} s_k = 0$ for all $j$ and and all $\al$---i.e. that the sections
$s_j$ are holomorphic. Further computations show that $D_{\ov{s}_j}(t_\al) =
J[H_j, V_\al] + i[H_j, V_\al]$.  Inspection of the Christoffel symbols shows
that for all $\al \neq \b$, there exist smooth functions $h^\g$, with $h^\b =
0$, such that $D_{\ov{t}_\al}(t_\b) = h^\g t_\g$.  Furthermore,
$D_{\ov{t}_\al}(t_\al) = -2s_\al =: -2 \sum_k s_k$.  Taken together, these facts
imply that $D_{\ov{t}_\al}(\rho) = 0$ for all $\al$ and that $D_{\ov{s}_j} \rho
= c_j \rho$ for some constant $c_j$. This means that the function $f$ must
satisfy the requirement that $\ov{s}_j(f) = -c_j f$ for all $j$ and that
$\ov{t}_\al(f) = 0$ for all $\al$. If we assume that $f$ to be $H$-invariant,
what we need is for $2y^j \frac{\pa f}{\pa y^j} = -c_jf$. The function $f =
e^{-\frac12 \sum_l c_l \log (y^l)}$ satisfies these requirements.
%
%----------------------------NEW LINE---------------------------%

%----------------------------NEW LINE---------------------------%
%
To compute $c_j$ we need to better understand $J[H_j, V_\al]$.  We have that
$J[H_j, V_\al] = -J\al(H_j) W_\al = \al(H_j) V_\al.$ Hence we have that
$D_{\ov{s}_j} t_\al = \al(H_j) t_\al$. This means that $c_j = \sum_{\al \in
\D^+} \al(H_j)$. But we have that $\sum_{\al \in \D^+} \al = 2\rho$, where
$\rho$ is the Weyl vector. Since $\rho(H_j) = 1$, for all $j$, we have $c_j =
2$, for all $j$. Hence we have
\begin{equation}
    \label{f-def}
    f = e^{-\sum_l \log(y^l)}.
\end{equation}
%
%
%----------------------------NEW LINE---------------------------%

%----------------------------NEW LINE---------------------------%
%
The dual of this form gives us a candidate for $\chi$. This means that we can
choose $\eta$ to be
\begin{equation*}
    \label{}
    \eta = f^{-2} \left( \bigwedge_{j=1}^n dY_j \wedge dH_j \right)
        \wedge \cdots \wedge \left( \bigwedge_{\al \in \D^+} dV_\al \wedge
        dW_\al \right).
\end{equation*}
If we pull this form back to $G \times P$ and write it in $(\frac{\pa}{\pa
x},H,V,W)$-frame, we get
\begin{equation}
    \label{eta_Y}
    \eta = \det(u_{jk}) f^{-2} \left( \bigwedge_{j=1}^n dx^j \wedge
        dH_j \right) \wedge \cdots \wedge \left( \bigwedge_{\al \in \D^+} dV_\al
        \wedge dW_\al \right).
\end{equation}

\subsection{Finding $\O^{N+n}$}
\label{finding_Omega}

Next we write \eqref{symplectic_Omega} in the $(\frac{\pa}{\pa x},H,V,W)$-frame
as
\begin{equation}
    \label{Omega_smart_coord}
    \O = dx^j \wedge dH_j + 2M^\al_jx^j dV_\al \wedge dW_\al,
\end{equation}
where $M^\al_j$ is given by \eqref{M_alpha-def}.  This means that up to a
multiplicative constant,
\begin{equation}
    \label{top_wedge_Omega}
    \O^{n+N} = p(x)\left( \bigwedge_j dx^j \wedge dH_j \right) \wedge \left(
        \bigwedge_\al dV_\al \wedge dW_\al \right),
\end{equation}
where $p(x)$ is the polynomial given by
    \begin{equation}
        \label{p-def}
        p(x) = \prod_{\al \in D^+} M_j^\al x^j.
    \end{equation}
Furthermore, \eqref{top_wedge_Omega} allows us to identify the
Duistermaat-Heckman polynomial. That polynomial is given by the pushfoward of
the volume form $\frac{\O^{n+N}}{(n+N)!}$ to $P$ under the moment map---i.e. one
needs to integrate \eqref{top_wedge_Omega} over $G$ which leaves an $n$-form on
$P$. Since the $H, V,$ and $W$-fields are all $G$-invariant, one easily sees
that this pushforward is given by $Cp(x)dx^1 \wedge \cdots \wedge dx^n$, where
$C$ is some positive constant. This gives us the relation $W(x) = Cp(x)$ .

\subsection{Finding $\D_{\ti{g}}$}

Our computations show that the function $\log|\frac{\O^{n+N}}{\eta}|$ is
independent of $G$. This means that we only need to compute the Laplacian for
functions which are $G$-invariant. Consider once again the form of $\ti{g}$
given by \eqref{g_G_x_P} on $G \times P$. In the $(\frac{\pa}{\pa
x},H,V,W)$-frame we see that $\sqrt{|\det(\ti{g})|} = C p(x)$, where $p$ is
given by \eqref{p-def} and $C$ is a positive constant.  Hence the Laplacian
$\D_{\ti{g}}$ on $G$-invariant functions $h$ is given by
\begin{equation}
    \label{laplacian}
    \D_{\ti{g}}(h) = -\frac{1}{p(x)} \frac{\pa}{\pa x^k} \left( p(x) u^{jk}
        \frac{\pa h}{\pa x^j} \right).
\end{equation}

\subsection{Scalar Curvature Equation}

Combining the results of the previous subsections, the scalar curvature $S$ is
given by
\begin{equation*}
    \label{}
    S = -\frac 12 \frac{1}{p(x)} \frac{\pa}{\pa x^k} \left( p(x) u^{jk} \frac{\pa}{\pa
        x^j} \left( \log p(x) - \log \det(u_{ab}) + 2 \log(f) \right)
        \right).
\end{equation*}
First note that 
\begin{equation}
    \label{eq2}
    -\frac1{p(x)} \frac{\pa}{\pa x^k} \left(p(x) u^{jk} \frac{\pa}{\pa x^j}
    (\log p(x)) \right) = -p^{-1}(u^{jk})_k p_j - p^{-1}u^{jk}p_{jk}.
\end{equation}
Next note that
\begin{equation}
    \label{eq3}
    -\frac{1}{p(x)} \frac{\pa}{\pa x^k} \left( p(x) u^{jk} \frac{\pa}{\pa
        x^j} \left( \log \det(u_{ab}) \right) \right) = p^{-1}p_k u^{jk} u^{ab}
        u_{abj} + (u^{jk} u^{ab} u_{abj})_k.
\end{equation}
Finally note that in \eqref{f-def} $f$ is written in the $\frac{\pa}{\pa
y}$-frame. In the $\frac{\pa}{\pa w}$-frame,
\begin{equation*}
    \label{}
    2\log(f) = -2\sum_l \log(y^l) = -4 \sum_l w^l.
\end{equation*}
Furthermore, the operator $u^{jk} \frac{\pa}{\pa x^j}$ transforms to
$\frac{\pa}{\pa w^k}$ and hence $u^{jk}\frac{\pa}{\pa x^j}(2 \log(f)) = -4.$
This means that
\begin{equation}
    \label{eq1}
    -\frac 12 \frac{1}{p(x)} \frac{\pa}{\pa x^k} \left( p(x) u^{jk}
        \frac{\pa}{\pa x^j} \left( 2 \log(f) \right) \right) = 2\sum_k
        \frac{\pa}{\pa x^k} \log p(x).
\end{equation}
If we sum \eqref{eq2} and \eqref{eq3}, we get $-p^{-1}(pu^{jk})_{jk}$. Hence the
scalar curvature is given on the polytope by the equation
    \begin{equation}
        \label{scalar_curvature}
        S = -\frac12 p^{-1}(pu^{jk})_{jk} + f_G
    \end{equation}
where $f_G = 2\sum_k \frac{\pa}{\pa x^k} \log p(x).$

\section{$G$-equivariant test configurations}

In this section, we construct a $G$-equivariant test configuration for the pair
$(\cV, \mathscr{L})$ given by \eqref{def_cV}. First note that in order to
construct a test configuration for $(\cV, \mathscr{L})$, by definition, we need
for $\mathscr{L}$ to be ample.  This is not always the case---the positivity of
$\mathscr{L}$ is dependent upon the chosen action of $(\cpx^*)^n$ on $L$.  The
action of $(\cpx^*)^n$ on $L$ is determined by the position of the moment
polytope $P$ of $V$ in $\real^n$. Equation \eqref{Omega_smart_coord} shows that
the line bundle $\mathscr{L}$ is ample if and only if the polytope lies within
the positive Weyl chamber of $\real^n$---which corresponds to the positive
quadrant by our choice of basis. More details can be found in \cite{don-sym}.
For the rest of this section we will assume that $\mathscr{L}$ is positive.
%
%----------------------------NEW LINE---------------------------%

%----------------------------NEW LINE---------------------------%
%
In \cite{don-02}, Donaldson constructs a test-configuration in the toric setting
which we will adapt to our toric fibrations. The construction only needs to be
changed in small ways so the reader is refered to his paper \cite{don-02} for
more details. Let $f$ be a convex, continuous, piecewise-linear, rational
function defined on $\real^n$ and $R$ a fixed number such that $f(x) \leq R-1,$
for all $x \in P$. Define $Q$ to be the convex polytope in $\real^{n+1}$ given
by
\begin{equation*}
    \label{}
    Q = \{ (x,t) \in \real^n \times \real \ | \ x \in P \ \mathrm{and} \ 0 <
        t < R - f(x) \}.
\end{equation*}
$P$ can be identified with the ``bottom'' face of $Q$. Let $(V, L)$ be the toric
variety corresponding to $P$ and let $(W, I)$ be the (possibly singular) toric
variety corresponding to $Q$. Next define $G' := G \times S^1$ and use $G'$ to
construct a fibration $(\cW, \mathscr{I})$ from $(W, I)$ similar to the
construction of $(\cV, \mathscr{L})$. Let $i : \cV \to \cW$ be the canonical
embedding induced by the inclusion $\ov{P} \to \ov{Q}$ and note that $i$ map is
left $G$-equivariant.
\begin{proposition}
    \label{prop_test_config_construction}
    There is a $\cpx^*$-equivariant map $p : \cW \to \proj^1$ with
    $p^{-1}(\i) = i(\cV)$ such that the restriction of $p$ to $\cW \
    \backslash \ i(\cV)$ is a test configuration for
    $(\cV,\mathscr{L})$.
\end{proposition}
\begin{proof}
    As explained in section \ref{lie_theory_sec}, a basis for the sections of
    $\mathscr{I} \to \cW$ is given by $s_{\l,i,j}$ where $\l$ is a lattice point
    in $\ov{P} \cap \mathbb{Z}^n$, $0 \leq i \leq R - f(\l)$, and $1 \leq j \leq \dim
    H^0(G_\cpx \times_B L_\l)$, as in \eqref{dimension_formula}. Note that the
    action of $T'$ on sections $s_{\l,i,j}$ and $s_{\l,i+1,j'}$ only differs in
    the last component of $T' = T \times S^1$. Choose a point $p \in \cW$ where
    none of these sections vanish (this corresponds to the open
    $(\cpx^*)^{n+1}$-torus in $W$). Next rescale the sections to all take the
    same value in $\mathscr{I}$ over the point $p$. Define the map $p:\cW \to
    \proj^1$ by
    \begin{equation*}
        x \mapsto [s_{\l,i,j}(x):s_{\l,i+1,j'}(x)].
    \end{equation*}
    As in Donaldson's case, this gives a $\cpx^*$-equivariant map $\cW \to
    \proj^1$, maping $i(\cV)$ to $[1,0].$ Define $\cX = \cW - i(\cV)$ and we
    have that $\mathscr{I}|_{\cX} \to \cX \to \cpx, x \mapsto \frac{s_{\l, i,
    j}(x)}{s_{\l, i+1, j'}(x)} \in \cpx,$ is a test configuration for $\cV$.  The
    rest of the proof goes through unchanged from the arguments in
    \cite{don-02}.
\end{proof}

\section{Futaki Invariant}

In this section, we will compute the Futaki invariant of the test-configuration
constructed in Proposition \ref{prop_test_config_construction}, hence providing
a proof of Theorem \ref{main_theorem}. Notation from section
\ref{lie_theory_sec} will be used throughout. To compute the Futaki
invariant---given by the number $F_1$ in \eqref{def_futaki_invariant}---we need
to compute $d_k$ and $w_k$. The lemmas in \cite{don-02} generalize
straightforwardly:
\begin{lemma}
    The number $d_k = h^0(\cX_0,\mathscr{I}|_{\cX_0}^k)$ equals
    $h^0(\cV,\mathscr{L}^k)$.
\end{lemma}
\begin{lemma}
    The sections $s_{\l,R-f(\l),j}$ are not identically zero when restricted to
    $\cX_0$ while all other sections restrict identically to zero. Consequently,
    the number $w_k$ is given by the sum of the weights on the sections
    $s_{\l,R-f(\l),j},$ for $1 \leq j \leq \dim (L_\l),$ and each weight is
    $f(\l)-R$.
\end{lemma}
Equation \eqref{G_sections} tells us that
\begin{equation*}
    \label{}
    d_k = \sum_{\l \in k\ov{P} \cap \mathbb{Z}^n} \dim H^0(G_\cpx \times_B L_\l)
\end{equation*}
and that the number $w_k$ is given by
\begin{align*}
    \label{}
    w_k &= \sum_{\l \in k\ov{P} \cap \mathbb{Z}^n} \dim H^0(G_\cpx \times_B
            L_\l) k(f(\l/k)-R) \\
        &= \sum_{\l \in k\ov{Q} \cap \mathbb{Z}^{n+1}} \dim H^0(G_\cpx \times_B
                L_\l) - \sum_{\l \in k\ov{P} \cap \mathbb{Z}^n} \dim H^0(G_\cpx
                \times_B L_\l).
\end{align*}
where $\pi : \mathbb{Z}^{n+1} \to \mathbb{Z}$ is the projection map
$(\l^1,\cdots,\l^{n+1}) \mapsto (\l^1,\cdots,\l^n),$ given in coordinates.
Equation \eqref{dimension_formula} says
\begin{equation*}
    \label{}
    \dim H^0(G_\cpx \times_B L_{(\sum_i \l^i \nu^i)}) = \prod_{\al \in \D^+} \frac{\k(\r
        + \sum_i \l^i\nu^i, \al)}{\k(\r, \al)}.
\end{equation*}
The important facts we need from Lie theory are that $\k(\nu^j,\al_k) =
\d_{jk}$, and that $\al = \sum_k M^\al_k \al_k$ as given by \eqref{M_alpha-def}. Hence we can
write the dimension formula as
\begin{align*}
    \label{}
    \dim H^0(G_\cpx \times_B L_{\l_i \nu^i}) &= \prod_{\al \in \D^+}
            \frac{\k(\sum_j (1+\l^j)\nu^j,\al)}{\k(\nu^1 + \cdots + \nu^n, \al)} \\
        &= \prod_{\al \in \D^+} \frac{\k(\sum_j (1+\l^j)\nu^j, \sum_k M^\al_k\al_k)}{\k(\nu^1 +
            \cdots + \nu^n, \sum_k M^\al_k\al_k)} \\
        &= \prod_{\al \in \D^+} \frac{(\sum_{j=1}^n M^\al_j) +
            \l^jM^\al_j}{\sum_{j=1}^n M^\al_j}.
\end{align*}
For notational convenience, define $|M^\al| = \sum_{j=1}^n M^\al_j$. Hence the
two numbers we need to understand are
\begin{equation*}
    d_k = \sum_{\l \in k\ov{P} \cap \mathbb{Z}^n} \prod_{\al \in \D^+} \frac{|M^\al|
            + \l^j M^\al_j}{|M^\al|}
\end{equation*}
and
\begin{equation*}
    w_k = \sum_{\l \in k\ov{Q} \cap \mathbb{Z}^{n+1}} \prod_{\al \in \D^+}
    \frac{|M^\al| + \l^jM^\al_j}{|M^\al|} - \sum_{\l \in k\ov{P} \cap
            \mathbb{Z}^n} \prod_{\al \in \D^+} \frac{|M^\al| +
            \l^jM^\al_j}{|M^\al|}.
\end{equation*}
We are interested in the ratio $\frac{w_k}{kd_k}$ and hence the common factor of
$|M^\al|$ in the denominator of these formulas can be ignored. This leads us to
define the polynomial $q(\l)$ by
\begin{equation*}
    q(\l) = \prod_{\al \in \D^+} \left( |M^\al| + \l^j M^\al_j \right).
\end{equation*}
Note that $q(\l)$ is an $N^\mathrm{th}$ degree polynomial in $\l$---where $N$
is the number of positive roots. We are interested in the assymptotics of such
polynomials as they are summed over lattice points in the polytope. These
assymptotics can be understood by using a specific measure on the boundary of
the polytope. We recall the following definition made in \cite{don-02}:
\begin{definition}
    \label{sigma_def}
    Let $P$ be an integer lattice polytope in $\real^n$. This means that for
    each face $F$ of $\ov{P}$, there is a vector $v_F$ which is perpendicular to
    $F$, pointing inwards, such that $v_F$ is the smallest such vector in the
    $\mathbb{Z}^n$ lattice. Let $l_F$ be the affine linear map such that
    $l_F^{-1}(0) \cap \ov{P} = F$ and such that the derivative of $l_F$ is equal
    to $v_F$. Finally, define the measure $d\si$ on $\pP$ by requiring that
    $d\si_F := d\si|_F$ be positive and that it satisfy $d\si_F \wedge dl_F =
    d\mu$, up to sign, where $d\mu$ is the standard Lebesgue measure on
    $\real^n$.
\end{definition}
We now wish to apply Lemma \ref{sigma_lemma} to this problem. 
To apply this lemma to $q$, we need to decompose $q$ into its homogeneous parts.
Let $q_k$ be the homogenous part of $q$ of order $k$.
\begin{align*}
    q(\l)  &= \prod_{\al \in \D^+} \left( |M^\al| + \l^j M^\al_j \right) \\
            &= \left( \prod_{\al \in \D^+} \l^j M^\al_j \right) + \left(
                \sum_{\b \in \D^+} |M^\b| \prod_{\al \neq \b} \left(
                \l^j M^\al_j \right) \right) + r(\l) \\
            &= q_N(\l) + q_{N-1}(\l) + r(\l),
\end{align*}
where $r$ is a polynomial of degree $N-2$. Note that $q_N$ is convex in the
positive quadrant. Using our lemma, we compute
\begin{align*}
    d_k &= \sum_{\l \in k\ov{P} \cap \mathbb{Z}^n} \left( q_N(\l) + q_{N-1}(\l) +
            r(\l) \right) \\
            &= k^N \sum_{\l \in \ov{P} \cap \frac{1}{k}\mathbb{Z}^n} q_N(\l) + k^{N-1}
            \sum_{\l \in \ov{P} \cap \frac{1}{k}\mathbb{Z}^n} q_{N-1}(\l) +
            \sum_{\l \in \ov{P} \cap \frac{1}{k}\mathbb{Z}^n} r(k\l) \\
        &= k^{N+n} \int_P q_N d\mu + k^{N+n-1} \left( \int_P q_{N-1}d\mu +
            \frac{1}{2} \int_\pP q_N d\si \right) + o(N+n-2).
\end{align*}
Similarly, we compute $w_k$:
\begin{align*}
    w_k &= \sum_{\l \in k\ov{Q} \cap \mathbb{Z}^{n+1}} q(\pi(\l)) - \sum_{\l \in
k\ov{P} \cap \mathbb{Z}^n} q(\l) \\
        &= k^{N+n+1} \left( \int_Q q_Nd\mu \right) + k^{N+n} \left( \int_Q
            q_{N-1}d\mu - \int_P q_Nd\mu + \frac{1}{2} \int_{\partial Q} q_Nd\si
            \right) \\
        & \ \ \ \ \ + o(K+n-1).
\end{align*}
The Fubini Theorem tells us that $\int_Qq_Nd\mu = \int_Pq_N(R-f)d\mu$ and that
$\int_Q q_{N-1}d\mu = \int_P q_{N-1}(R-f)d\mu.$ Furthermore,
\begin{equation*}
    \label{}
    -\int_Pq_Ndx + \frac{1}{2} \int_{\partial Q}q_Nd\si = \frac{1}{2}\int_\pP
        q_N(R-f)d\si.
\end{equation*}
Hence we have that
\begin{equation*}
    \label{}
    d_k = Ck^{N+n} + Dk^{N+n-1} + o(N+n-2)
\end{equation*}
and
\begin{equation*}
    \label{}
    w_k = Ak^{N+n+1} + Ck^{N+n} + o(N+n-1),
\end{equation*}
where the constants $A, B, C,$ and $D$ are given by:
\begin{itemize}
    \item[] $A = \int_P q_N(R-f)d\mu$
    \item[] $B = \int_P q_{N-1}(R-f)d\mu + \frac{1}{2}\int_\pP q_N(R-f)d\si$
    \item[] $C = \int_P q_Nd\mu$
    \item[] $D = \int_P q_{N-1}dx + \frac{1}{2}\int_{\pP} q_N d\si$
\end{itemize}
To compute the Futaki invariant, we need to compute the term $F_1 =
C^{-2}(BC-AD)$. Straight-forward computations yield
\begin{align*}
    F_1 &= \frac{-1}{\int_P q_Nd\mu} \left\{ \int_P f q_{N-1} dx +
            \frac{1}{2}\int_\pP f q_N d\si - \frac{\int_P q_{N-1} d\mu +
            \frac{1}{2} \int_\pP q_N d\si}{\int_P q_N d\mu}\int_P f q_N d\mu
            \right\}.
\end{align*}
%
%----------------------------NEW LINE---------------------------%

%----------------------------NEW LINE---------------------------%
%
First note that $q_N$ is the same polynomial as $p$ given by \eqref{p-def}.
Next note that $q_{N-1} = \sum_l \frac{\pa}{\pa x^l} p = \frac14 p f_G$, where
$f_G$ is given by \eqref{scalar_curvature}.
\begin{lemma}
    We have that
    \begin{equation*}
        \label{}
        \frac{\int_P q_{N-1} d\mu + \frac{1}{2} \int_\pP q_N d\si}{\int_P q_N
            d\mu} = \frac{a}{2},
    \end{equation*}
    where $a$ is the average scalar curvature of any metric.
\end{lemma}
\begin{proof}
    Let $u$ be the symplectic potential of any metric. Then
    \begin{align*}
        \label{}
        a \int_P p(x)dx &= \int_P S p(x)dx \\
            &= \frac 12 \int_P -(p(x)u^{jk})_{jk} dx + \int_P f_G p(x) dx \\
            &= \frac 12 \int_\pP p(x)2d\si + 2\int_P q_{N-1} dx \\
            &= 2 \left( \frac12 \int_\pP p(x)d\si + \int_P q_{N-1} dx \right). \\
    \end{align*}
    Which is what we needed to show.
\end{proof}
Hence we have proved that the Futaki invariant of the test configuration we constructed is equal to
    \begin{equation*}
        \label{}
        -\frac{1}{2\mathrm{Vol}_p(P)} \left( \int_P ff_Gpd\mu +
            \int_{\partial P} fpd\si - a\int_P fpd\mu \right).
    \end{equation*}
As explained in Subsection \ref{finding_Omega}, $p(x) = CW(x)$ for some positive
constant $C$, and hence the proof of Theorem \ref{main_theorem} is complete.

\section{Mabuchi Functional}

This section straightforwardly generalizes Donaldson's work, and the reader is
refered to \cite{don-02} for more details. Given our scalar curvature equation
\eqref{scalar_curvature}, we consider the analytic picture as follows. Let $l_F$
be the affine linear function on $P$ as in Definition \ref{sigma_def}. Next
define $u_\si$ to be the function
\begin{equation*}
    \label{}
    u_\si(\cdot) = \frac12 \sum_F l_F(\cdot) \log l_F(\cdot).
\end{equation*}
Let $\cS$ be the space of all $u$ defined on $P$ such that $u-u_\si$ is smooth
up to the boundary of $\ov{P}$ and such that $u$ is strictly convex on $P$ as well as
when restricted to any of the faces of $\ov{P}$. The goal is then to solve the
equation
\begin{equation}
    \label{analytic_equation}
    - W^{-1}(Wu^{ij})_{ij} = A,
\end{equation}
where $A$ is a smooth function on $\ov{P}$. Note that if one chooses $A = \frac{a -
f_G}2$, then this is the constant scalar curvature equation. Next define the
functional $\cL_A$ and $\cF_A$ on $\cS$ by
\begin{equation*}
    \label{cL_def}
    \cL_A(u) = 2\int_\pP uWd\si - \int_PAuWd\mu,
\end{equation*}
and
\begin{equation*}
    \label{cF_def}
    \cF_A(u) = -\int_P \log \det (u_{jk})W d\mu + \cL_A(u).
\end{equation*}
$\cF_A$ is a concave functional on $\cS$. The variation $\d \cF_A$ of $\cF_A$
by a smooth function $\d u$ is given by
\begin{equation*}
    \label{}
    \d \cF_A = - \int_P u^{jk} \d u_{jk} W d\mu + \cL_A(\d u).
\end{equation*}
Donaldson's work in \cite{don-02} allows us to to use the boundary conditions of
$u$ to integrate by parts twice to get
\begin{equation*}
    \label{}
    - \int_P u^{jk} \d u_{jk} W d\mu = -\int_P W^{-1}(Wu^{jk})_{jk} \d u W d\mu
        - 2\int_\pP \d u W d\si,
\end{equation*}
and hence we have that
\begin{equation*}
    \label{}
    \d \cF_A = - \int_P u^{jk} \d u_{jk} W d\mu - \int_P A \d u W d\mu.
\end{equation*}
Since $W$ is strictly positive on $P$, this says that solutions to
\eqref{analytic_equation} are the same as critical points of the concave
functional $\cF_A$. If we choose $A = \frac{a-f_G}2$, then $\cF_A$ is the
Mabuchi functional on the polytope $P$ and we arive at a proof of Thereom
\ref{mabuchi_theorem}.

\section{Proof Lemma \ref{sigma_lemma}}
\label{section_lemma_proof}

Definition \ref{sigma_def} was stated as is to agree with the definition of
$\si$ given in \cite{don-02} to avoid possible confusion. However, for our
purposes, a more useful and equivalent definition is given by the following
lemma.
\begin{lemma}
    \label{sigma_def2_lemma}
    The measure $d\si(F)$ of a face $F$ of the polytope $P$ is given by
    \begin{equation}
        \label{sigma_def2}
        d\si(F) = \lim_{k \to \i} \frac{\#(F \cap
                    \frac{1}{k}\mathbb{Z}^n)}{k^{n-1}},
    \end{equation}
where $\#(S)$ is the number of points in the set $S$.
\end{lemma}
\begin{proof}
    We can assume that $P$ is given as the convex hull of the extreme points
    $(0,\ldots,0),$ $(p_1, 0, \ldots, 0), \ldots, (0, \ldots, 0, p_n),$ where
    the $p_i$ are positive integers and $p_i$ and $p_j$ are coprime for $i \neq
    j$. To verify the lemma, we only need to show that the equation is satisfied
    for the face $F$ of $P$ that does not include the origin. (The other faces
    are entirely contained in the standard subsets $(x_i \equiv 0)$ where this
    lemma is clearly true.) The primitive outward orthogonal vector $v$ to face
    $F$ is given by
    \begin{equation*}
        v = \sum_{i = 1}^n \left( \prod_{j \neq i} p_j \right) e_i,
    \end{equation*}
    where the $e_i$ are the standard basis vectors of $\real^n$. Hence the
    measure $d\si_F$ is given by the following form on $\real^n$ restricted to
    $F$:
    \begin{equation*}
        \label{}
        d\si_F = \left( \prod_{i=1}^{n-1} p_i \right)^{-1} dx_1 \wedge \cdots \wedge
                dx_{n-1}.
    \end{equation*}
    This form can be integrated over the face $P$ given by $(x_n \equiv 0)$ and
    yields the result
    \begin{equation}
        \label{F_measure}
        d\si_F(F) = \frac{1}{(n-1)!}.
    \end{equation}
%
%----------------------------NEW LINE---------------------------%

%----------------------------NEW LINE---------------------------%
%
    Next we would like to verify that we get the same result from equation
    \eqref{sigma_def2}. The fact that $p_i$ and $p_j$ are coprime for $i \neq j$
    tells us that the set $F \cap \mathbb{Z}^n$ has exactly $n$ points and that
    those are the extreme points of $F$. This means that the ``projection" map
    $\pi$ defined by
    \begin{equation*}
        \label{}
        \pi(x_1, \ldots, x_{n-1}, x_n) = \left( \frac{x_1}{p_1}, \ldots,
            \frac{x_{n-1}}{p_{n-1}} \right),
    \end{equation*}
    maps the lattice points of $F \cap \mathbb{Z}^n$ to the lattice points of
    the standard $(n-1)$-simplex $S$ in $\real^{n-1}$. This mapping shows that
    the number of lattice points in $F \cap \frac{1}{k} \mathbb{Z}^n$ is the
    same as the number of lattice points in $S \cap \frac{1}{k}
    \mathbb{Z}^{n-1}$. But the final number is simply $k^n \mathrm{Vol}(S)$ to
    highest order. One can verify that $\mathrm{Vol}(S)$ agrees with
    \eqref{F_measure} which proves the lemma.
\end{proof}
Equation \eqref{sigma_def2} says that the measure of $F$ is given asymptotically
by the number of lattice points in $kF$. Note that this makes it clear that the
measure is invariant under transformations in $GL(n,\mathbb{Z})$.
%
%----------------------------NEW LINE---------------------------%

%----------------------------NEW LINE---------------------------%
%
We will prove Lemma \ref{sigma_lemma} by comparing the sum on the left side of
\eqref{lemma_sum} to the integrals on the right side of the equation. This
requires some care and leads us to make quite a few definitions. Let $\cP_k$ be
the set of points $\cP_k = P \cap \frac{1}{k}\mathbb{Z}^n.$
For a given point $p \in \cP_k$, let $\Box_k(p) = \Box_k(p_1, \cdots, p_n)$ be
the \textit{box} defined by
    \begin{equation*}
        \label{}
        \Box_k(p_1, \ldots, p_n) = \left[ p_1, p_1 + \frac{1}{k} \right] \times 
                                \cdots \times \left[ p_n, p_n + \frac{1}{k} 
                                \right] \subset \real^n.
    \end{equation*}
Given a box $\Box_k(p)$, we will call $p$ the \textit{corner point} of
$\Box_k(p)$ and call $p_{k,m} := (p_1 + \frac{1}{2k}, \ldots, p_n +
\frac{1}{2k})$ the \textit{midpoint} of $\Box_k(p)$. Furthermore, we will need
to partition $\cP_k$ into the disjoint sets of \textit{interior}, \textit{face},
and \textit{exterior} points as follows:
    \begin{itemize}
        \item[] $\cI_k = \{ p \in \cP_k \ | \ \Box_k(p) \cap P = \Box_k(p) \}$
        \item[] $\cE_k = \{ p \in \cP_k \ | \ \Box_k(p) \cap P = \{ p \} \ \}$
        \item[] $\cF_k = \cP_k \ \backslash \ \left( \cI_k \cup \cF_k \right)$
    \end{itemize}
Note: $\cI_k$ contains points on the boundary of $P$. Next, define (non-convex)
subsets of $\real^n$ as follows
    \begin{itemize}
        \item[] $P_{\cI,k} = \bigcup_{p \in \cI_k} \Box_k(p)$
        \item[] $P_{\cF,k} = \bigcup_{p \in \cF_k} \Box_k(p)$
        \item[] $P_{\cE,k} = \bigcup_{p \in \cE_k} \Box_k(p)$
    \end{itemize}
Figure \ref{figure} illustrates these definitions.
\begin{figure}[h]
    \begin{center}
        \includegraphics{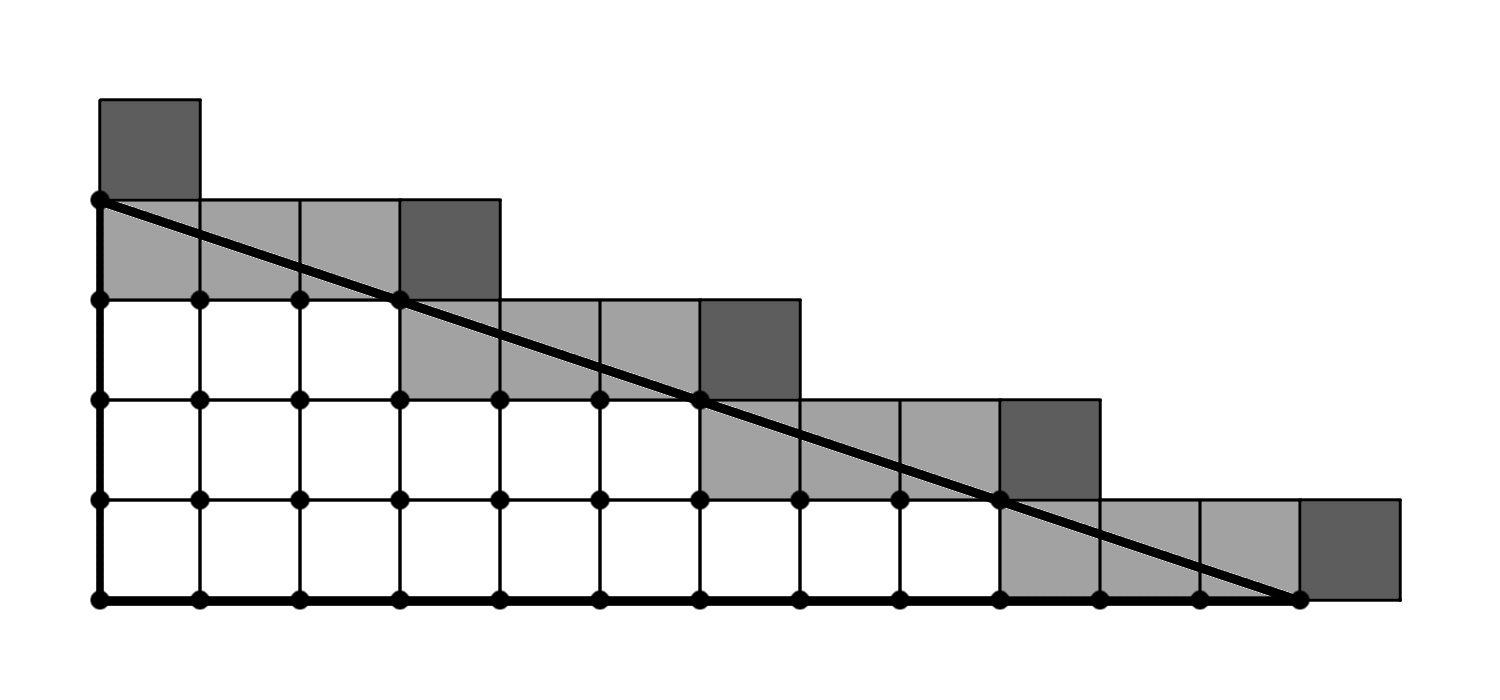}
        \caption{The polytope in this example is a triangle with height 1 and
            base 3. Furthermore, $k$ = 4. The dots correspond to lattice points
            in $\cP_3$. The white squares correspond to the set $P_{\cI,3}$, the
            light gray squares correspond to the set $P_{\cF,3}$ and the dark
            gray squares correspond to the set $P_{\cE,3}$.}
        \label{figure}
    \end{center}
\end{figure}
\begin{lemma}
    \label{}
    Let $h$ be a $C^2$ function on $B = \Box_k(0, \ldots, 0)$ and $p_{k,m}$ the
    midpoint of $B$. Then
    \begin{equation*}
        \label{}
        \left| k^n \left( \int_B h d\mu \right) - h(p_{k,m}) \right| \leq
            \frac{1}{k^2} C_n ||h||_{C^2(B)},
    \end{equation*}
    where $C_n$ only depends on the dimension $n$.
\end{lemma}
\begin{proof}
    First consider the one-dimensional case where $B = [0, \frac{1}{k}]$.
    Integrating by parts, we see
    \begin{align*}
        \int_0^{\frac{1}{k}} h(x)dx &= \int_0^{\frac{1}{2k}} h(x)dx +
                \int_{\frac{1}{2k}}^{\frac{1}{k}} h(x)dx \\
            &= \int_0^{\frac{1}{2k}} \left( \frac{x^2}{2} + Ax + B \right)
                h''(x) dx - \left. \left. \left( \frac{x^2}{2} + Ax + B \right)
                h'(x) \right|^{\frac{1}{2k}}_0 + (x+A) h(x)
                \right|^{\frac{1}{2k}}_0 \\
            &+ \int_{\frac{1}{2k}}^{\frac{1}{k}} \left( \frac{x^2}{2} + Cx + D
                \right) h''(x) dx - \left. \left. \left( \frac{x^2}{2} + Cx + D
                \right) h'(x) \right|_{\frac{1}{2k}}^{\frac{1}{k}} + (x+C) h(x)
                \right|_{\frac{1}{2k}}^{\frac{1}{k}},
    \end{align*}
    where $A, B, C,$ and $D$ are constants that we can choose freely. By choosing
    $A = B = 0$, $C = -\frac{1}{k}$, and $D = \frac{1}{2k^2}$, we see that
    \begin{equation*}
        \label{}
        \int_0^\frac{1}{k} h(x)dx = \frac{1}{k} h \left( \frac{1}{2k} \right) +
            \int_0^\frac{1}{2k} \frac{x^2}{2}h''(x)dx +
            \int_\frac{1}{2k}^\frac{1}{k} \frac{\left( x - \frac{1}{k}
            \right)^2}{2} h''(x)dx.
    \end{equation*}
    Hence we have
    \begin{equation*}
        \label{}
        \left| k \int_0^\frac{1}{k} h(x)dx - h \left( \frac{1}{2k} \right)
            \right| \leq \frac{1}{24k^2}\max_{x \in \left[ 0, \frac{1}{k}\
            \right]} |h''(x)|.
    \end{equation*}
%
%----------------------------NEW LINE---------------------------%

%----------------------------NEW LINE---------------------------%
%
    The proof is completed by induction. Assume the lemma is true for the
    $n$-dimensional case. Let $h$ be a function of $n+1$ variables. Define
    $\tilde{h}(x) = h(x_1, \ldots, x_n, x)$ and apply the previous argument and
    the induction hypothesis to get the desired result.
\end{proof}
\begin{lemma}
    \label{}
    Let $h$ be a $C^2$ function on $P$ and $p_{k,m}$ the midpoint of box
    $\Box_k(p)$. Then we have
    \begin{equation*}
        \label{}
        \left| k^n \int_{P_{\cI,k}} h(x) dx - \sum_{p \in \cI_k} h(p_{k,m})
            \right| \leq k^{n-2} C_n K_P ||h||_{C^2(P)|},
    \end{equation*}
    where $C_n$ is the same constant as the last lemma, and $K_P$ is a constant
    depending on the geometry of $P$.
\end{lemma}
\begin{proof}
    Sum up the previous lemma over the points in $\cI_k$.
\end{proof}
%
%----------------------------NEW LINE---------------------------%

%----------------------------NEW LINE---------------------------%
%
These lemmas yield a sort of asymptotic estimate for
\begin{equation*}
    \label{}
    k^n \int_{P_{\cI,k}} h(x)dx,
\end{equation*}
but in order to estimate the right side of \eqref{lemma_sum}, we still need an
estimate for
\begin{equation}
    \label{eq_bound_int}
    k^n \int_{P \ \backslash \ P_{\cI,k}} h(x)dx.
\end{equation}
We will compare \eqref{eq_bound_int} to the sum
\begin{equation}
    \label{eq_face_sum}
    \sum_{p \in \cF_k} h(x_{p,k}),
\end{equation}
where $x_{p,k}$ is some arbitrary point in $B = \Box_k(p)$. Note that if
$x_{p,k}, x_{p,k}' \in \Box_k(p)$, then
\begin{equation}
    \label{face_sum_error}
    |h(x_{p,k}) - h(x_{p,k}')| \leq \frac{\sqrt{n}}{k}||h||_{C^1(B)}.
\end{equation}
Now let $m_k(p) \in \Box_k(p)$ be such that
$\min_{\Box_k(p)} h = h(m_k(p))$ and define $M_k(p)$ similarly to be where $h$
takes its maximum. We have then that
\begin{equation}
    \label{face_min_max}
    \sum_{p \in \cF_k} h(m_k(p)) \leq k^n \int_{P_{\cF,k}} h(x)dx \leq \sum_{p \in
        \cF_k} h(M_k(p)).
\end{equation}
\begin{lemma}
    \label{}
    There exists a constant $C$ depending only on the dimension $n$, the
    geometry of $P$, and $||h||_{C^1(P)}$ such that
    \begin{equation*}
        \label{}
        \left| k^n \int_{P \ \backslash \ P_{\cI,k}} hd\mu - \frac{1}{2}
            \sum_{p \in \cF_k} h(x_{p,k}) \right| \leq Ck^{n-2}.
    \end{equation*}
    where, as before, $x_{p,k}$ is any point in $\Box_k(k)$.
\end{lemma}
\begin{proof}
    Given \eqref{face_min_max} and \eqref{face_sum_error}, we need only show
    \begin{equation}
        \label{face_comparison_integral}
        \left| k^n \int_{P \ \backslash \ P_{\cI,k}} hd\mu - \frac{1}{2}
            k^n \int_{P_{\cF,k}} hd\mu \right| \leq Ck^{n-2}.
    \end{equation}
    The idea of this proof is the following observation: Assume we are given a
    rational plane $H \subset \real^n$ through the origin which cuts $\real^n$ into
    two pieces $S_1$ and $S_2$. Furthermore, assume we are given a hypercube $B =
    \Box_1(p)$ such that $H$ intersects the interior of $B$. If $B'$ is the
    hypercube given by reflecting $B$ about the origin, then the pair $(B,B')$
    has the property that $\mathrm{Vol}(B \cap S_1) + \mathrm{Vol}(B' \cap S_1)
    = 1$. We will use this idea to prove \eqref{face_comparison_integral} by
    considering each of the different lattice points in $\cF_k$ as our
    ``origin".
%
%----------------------------NEW LINE---------------------------%

%----------------------------NEW LINE---------------------------%
%
    We would ideally like to proceed as follows. Let $p \in \cF_k$ be a lattice
    point and let $q \in \cE_k$ be the unique point in $\cE_k$ which is closest
    to $B = \Box_k(p)$. Then let $p'$ be the corner point of the box $B'$ given
    by reflecting $B$ about the point $q$.
%
%----------------------------NEW LINE---------------------------%

%----------------------------NEW LINE---------------------------%
%
    There are two problems with this approach. The first is that the
    corresponding point $p'$ may not lie in $\cF_k$. This will be true for the
    boxes $B$ which are close to the boundary of $F$. We deal with this problem
    by not considering the points $p$ which have no corresponding point. The
    second problem is that the point $q$ need not actually be unique. This could
    be handled multiple ways, but the easiest seems to be to do the following:
    Let $d$ be the distance from $B$ to the lattice $\cE_k$. Let $Q_B$ be the
    set of $q \in \cE_k$ such that $\mathrm{dist}(B,q) = d$. Let $N_B$ the
    number of elements in $Q_B$. Finally consider $N_B$ pairs $(B, B')$---one
    for each different $q \in Q_B$---and in the end weight each pair by the
    fraction $\frac{1}{N_B}$. This allows us to compare the two integrals in
    \eqref{face_comparison_integral} with the desired precision.
    \end{proof}
Taken together, these lemmas result in the following:
\begin{lemma}
    \label{}
    There is a constant $C$ depending only on the geometry of $P$,
    $||h||_{C^2}$, and the dimension $n$ so that
    \begin{equation}
        \label{midpoint_approx}
        \left| k^n \int_P h(x) - \left( \sum_{p \in \cI_k} h(p_{k,m}) + \frac{1}{2}
            \sum_{p \in \cF_k} h(x_{p,k}) \right) \right| \leq Ck^{n-2},
    \end{equation}
    where $x_{p,k}$ is an arbitrary point in $\Box_k(p)$ as before.
\end{lemma}
%
%----------------------------NEW LINE---------------------------%

%----------------------------NEW LINE---------------------------%
%
We are finally in the position to prove Lemma \ref{sigma_lemma}.
\begin{proof}
    To prove this we may assume that $P$ is a ``stretched standard simplex".
    I.e. that there is a vertex $v$ of $P$, such that if one chooses $v$ as the
    origin, then $P$ is given as the convex hull of the origin $v$ and the
    points $p_1e_1, \ldots, p_ne_n$, where $p_i > 0$ and $e_i$ is the standard
    basis vector. Any polytope $P$ can be deconstructed into such stretched
    standard simplices and then if one applies Lemma \ref{sigma_lemma} to each
    piece, one gets the result for all of $P$.
%
    %----------------------------NEW LINE---------------------------%

    %----------------------------NEW LINE---------------------------%
%
    The asymptotic sum $S_k$ we need to approximate is given by
    \begin{equation}
        \label{Sk_def}
        S_k = \sum_{p \in \cP_k} h(p) = \sum_{p \in \cI_k} h(p) + 
                \sum_{p \in \cF_k} h(p) + \sum_{p \in \cE_k} h(p).
    \end{equation}
    The main idea of the proof is to compare \eqref{Sk_def} with the ``midpoint
    rule" for integrals. Given \eqref{midpoint_approx}, we only need to
    understand the asymptotics of the difference $S_k - M_k$, where
    \begin{equation}
        \label{midpoint_sum_def}
        M_k = \left( \sum_{p \in \cI_k} h(p_{k,m}) + \frac{1}{2} \sum_{p \in
            \cF_k} h(p_{k,m}) \right).
    \end{equation}
    Now let $p$ be any lattice point in $\pP \ \backslash \ F$. Let $l_{p,k}(j) =
    p + \frac{1}{2k}(j, \ldots, j).$ I.e. $l_{p,k}(0) = p$, $l_{p,k}(1) =
    p_{k,m}$, etc. Let $L_{p,k} = l_{p,k}(\real)$ be the line through $p$
    parallel to the vector $(1, \ldots, 1)$.
%
    %----------------------------NEW LINE---------------------------%

    %----------------------------NEW LINE---------------------------%
%
    Now for each $p$, we will define an alternating sum $A_{p,k}$ as follows. If
    $p \in \pP \ \cap F$, then define $A_p = h(p)$. Otherwise, if $L_{p,k} \cap
    P$ is a line segment connecting $p$ to an element of $\cE_k$, define
    $A_{p,k}$ as
    \begin{equation*}
        \label{}
        A_{p,k} = h(l_{p,k}(0)) - h(l_{p,k}(1)) + h(l_{p,k}(2)) - h(l_{p,k}(3))
            \pm \cdots + h(l_{p,k}(N)),
    \end{equation*}
    where we have $l_{p,k}(N) \in \cE_k$ is that final terminating point.
    Finally, if $p$ satisfies neither of the preceeding requirements, define
    \begin{equation*}
        \label{}
        A_{p,k} = h(l_{p,k}(0)) - h(l_{p,k}(1)) + h(l_{p,k}(2)) - h(l_{p,k}(3))
            \pm \cdots + h(l_{p,k}(N-1)) - \frac12 h(l_{p,k}(N)),
    \end{equation*}
    where $l_{p,k}(N)$ is the midpoint of the box $\Box_k(q)$ and $q$ is the
    point in $\cF_k$ which lies on the line $L_{p,k}$.
%
    %----------------------------NEW LINE---------------------------%

    %----------------------------NEW LINE---------------------------%
%
    With this setup, we have that
    \begin{equation*}
        \label{}
        S_k - M_k = \sum_{p \in \pP \ \backslash \ F} A_{p,k}.
    \end{equation*}
    Now if $L_{p,k} \cap P$ is a line terminating in a point in $\cE_k$, then we
    have that
    \begin{align*}
        \label{}
        A_{p,k} &= \frac12 h(l_{p,k}(0)) + \frac12 \Bigg\{ [h(l_{p,k}(0)) -
                h(l_{p,k}(1))] - [h(l_{p,k}(1)) - h(l_{p,k}(2))] \\
            & \ \ \ \  + \cdots - [h(l_{p,k}(N-1)) - h(l_{p,k}(N))]
                \Bigg\} + \frac12 h(l_{p,k}(N)). \\
    \end{align*}
    Due to convexity, the middle terms form an alternating series, and hence
    $A_{p,k} = \frac12 [h(l_{p,k}(0)) + h(l_{p,k}(N))] + \frac Ck,$ for some
    constant $C$ depending on the derivative of $h$. Going back to the case
    where $L_{p,k} \cap P$ does not terminate in a point in $\cE_k$, similar
    arguments show that $A_{p,k} = \frac12 h(l_{p,k}(0)) + \frac Ck$, with $C$
    once again depending upon $||h||_{C^1(P)}$.
%
    %----------------------------NEW LINE---------------------------%

    %----------------------------NEW LINE---------------------------%
%
    Combining these results we have that up to highest order $S_k - M_k =
    \frac12 \sum_{p \in \pP} h(p)$. If we apply Lemma \ref{sigma_def2_lemma},
    the proof of Lemma \ref{sigma_lemma} is complete.
\end{proof}

\noindent \textbf{Remark:} In the preceeding proof we essentially only used the
fact that $h$ is convex along lines parallel to the vector $(1, \cdots, 1)$.
One may be tempted to conclude that that is all that is necessary for Lemma
\ref{sigma_lemma}.  However, in the previous proof we assumed that $P$ was in
the form of a stretched standard simplex. If $P$ is arbitrary, we would need to
decompose it into stretched standard simplices and apply this result to each one
individually. On those other simplices, we would most likely have to change
orientations and consider lines that are going in other directions. Hence in
general we do need $h$ to be convex in all directions for Lemma
\ref{sigma_lemma} to be true.

{\noindent \footnotesize Department of Mathematics\\
Columbia University, New York, NY 10027\\
tomnyberg@math.columbia.edu\\}

\end{normalsize}
\end{document}